\documentclass[bj]{imsart}

\RequirePackage{amsthm,amsmath,amsfonts,amssymb}
\RequirePackage[numbers]{natbib}

\startlocaldefs

\usepackage{graphicx}

\usepackage[colorlinks=false,pdfborderstyle={/W 1}]{hyperref}

\newif\ifhyper\IfFileExists{hyperref.sty}{\hypertrue}{\hyperfalse}
\ifhyper\usepackage{hyperref}\fi

\def\arXiv#1#2{\href{http://arxiv.org/abs/#1}{{\tt arXiv:#1 [#2]}}}

\def\given{\,|\,} 
\def\bgiven{\;\big|\;}

\def\Var{\mathrm{Var}}

\newtheorem{theorem}{Theorem}
\newtheorem{question}[theorem]{Question}

\newtheorem{lemma}[theorem]{Lemma}
\newtheorem{proposition}[theorem]{Proposition}

\theoremstyle{remark}
\theoremstyle{definition}

\newcommand{\C}{\mathcal{C}}
\newcommand{\LL}{\mathcal{L}}
\newcommand{\Z}{\mathbb{Z}}
\newcommand{\N}{\mathbb{N}}

\def \eps {\epsilon}
\def \P {{\Bbb P}}
\def \E {{\Bbb E}}

\def\max{{\rm max}}

\newcommand{\TT}{\mathfrak{T}}
\newcommand{\T}{\mathcal{T}}

\newcommand{\CC}{\Pi}
\newcommand{\BB}{\Pi}

\newcommand{\R}{R}
\newcommand{\F}{F}
\newcommand{\FF}{\mathfrak{F}}
\newcommand{\Can}{{\rm Can}}

\def\Gstar{{{\cal G}_{*}}}

\def\omps{{\omega_\delta^\eps}}

\def\cE{\mathcal{E}}

\endlocaldefs

\begin{document}

\begin{frontmatter}
\title{Finite-energy infinite clusters without anchored expansion}
\runtitle{Finite-energy infinite clusters without anchored expansion}

\begin{aug}
\author[A]{\fnms{G\'abor} \snm{Pete}\ead[label=e1]{robagetep@gmail.com}}
\and
\author[B]{\fnms{\'Ad\'am} \snm{Tim\'ar}\ead[label=e2]{madaramit@gmail.com}}
\address[A]{Alfr\'ed R\'enyi Institute of Mathematics, Budapest, Hungary\\ 
and Budapest University of Technology and Economics, Budapest, Hungary\\
\printead{e1}}

\address[B]{University of Iceland, Reykjavik, Iceland\\ 
and Alfr\'ed R\'enyi Institute of Mathematics, Budapest, Hungary\\ 
\printead{e2}}
\end{aug}

\begin{abstract}
Hermon and Hutchcroft have recently proved the long-standing conjecture that in Bernoulli$(p)$ bond percolation on any nonamenable transitive graph $G$, at any $p > p_c(G)$, the probability that the cluster of the origin is finite but has a large volume $n$ decays exponentially in $n$. A corollary is that all infinite clusters have anchored expansion almost surely. They have asked if these results could hold more generally, for any finite energy ergodic invariant percolation. We give a counterexample, an invariant percolation on the 4-regular tree.
\end{abstract}


\end{frontmatter}


\section{Introduction}

This paper gives a negative answer to the following recent question on invariant bond percolations on nonamenable transitive graphs. We refer the reader to \cite{PGG} for background, but will briefly recall the basic definitions and motivations after the question. 

\begin{question}[Hermon and Hutchcroft, Question 5.5 in \cite{HH}]\label{q.HeHu}
Let $G$ be a nonamenable unimodular transitive graph, and let $\omega$ be an ergodic invariant bond percolation process. Apply an $\epsilon>0$ of Bernoulli noise to $\omega$ to get a new invariant percolation configuration $\omega'$; i.e., we take the symmetric difference of $\omega$ and a Bernoulli$(\eps)$ bond percolation.
\begin{enumerate}
\item[{\bf (A)}] If $\omega'$ has infinite clusters, must these infinite clusters have anchored expansion?  
\item[{\bf (B)}] Is the probability that the origin lies in a finite cluster of $\omega'$ of size at least $n$ exponentially small?
\end{enumerate}
\end{question}

A bounded degree infinite graph $G=(V,E)$ is called {\it nonamenable} if the boundary-to-volume ratio $|\partial_E K| / |K|$ stays above some $c>0$ for every finite subset $K \subset V(G)$ of the vertices, where $\partial_E K$ is the set of edges with one endpoint in $K$ the other in $K^c$. As a relaxation of this property, the graph is {\it anchored nonamenable}, or in other words, it has {\it anchored expansion}, if 
$$
\iota^*_o := \inf \left\{ \frac{|\partial K|}{|K|} : o \in K \subset V(G)\textrm{ connected finite sets} \right\} > 0
$$
holds for some (and then, for any) anchor $o \in V(G)$.

An {\it invariant bond percolation} on an infinite transitive graph $G$ is just a random subset of the edges whose distribution is invariant under the automorphism group of $G$. Some standard examples, beyond Bernoulli$(p)$ bond percolation \cite{BS}, are the free or wired infinite volume FK$(p,q)$ random cluster models \cite{Lyons}, random interlacements \cite{TT}, the edges spanned by the open vertices of any invariant site percolation model (e.g., the Ising model \cite{TomIsing}, or the super-level sets of an invariant height function, such as the discrete Gaussian Free Field \cite{GFF,AC}), or processes obtained by local modifications of the above (such as factor of iid percolations \cite{Lfiid,BV}). The references given here are somewhat ad hoc; we have tried to give papers that focus on these processes on transitive graphs beyond $\Z^d$.

That the {\it subcritical phase} of most of the above processes is well-behaved in the sense that correlations and/or cluster sizes decay exponentially fast on any transitive graph is relatively well-understood by now \cite{ABF,DCT,DCRT}. In the {\it supercritical phase}, we need a truncation, such as a conditioning on the cluster of the origin to be finite. And, the results are more subtle: even for Bernoulli percolation on amenable transitive graphs such as $\Z^d$, because of the vanishing boundary-to-volume ratio, the cluster size does not have an exponential decay; on the other hand, with some non-trivial ways to measure the size of the boundary, one can sometimes get an exponential decay for that, which is still very useful; see \cite{KZh,Pet,BST}. For non-amenable transitive graphs, the need for a subtle definition does not arise, but the proofs are still harder. Hermon and Hutchcroft \cite{HH} have only recently proved the long-standing conjecture (probably first stated explicitly in \cite{BST}) that for Bernoulli$(p)$ bond percolation on any nonamenable transitive graph $G$, at any $p > p_c(G)$, the answer to Question~\ref{q.HeHu}~(B) is affirmative. 

The notion of anchored expansion was first explicitly defined by Benjamini, Lyons and Schramm in \cite{pertu}; more general anchored isoperimetric inequalities appeared implicitly in \cite{Tho}, explicitly in \cite{Pet}. See Section 6.8 of \cite{LP} for further background. The motivation is the obvious theoretical and practical interest in the robustness of large-scale geometric properties of transitive graphs under reasonable random perturbations. Infinite clusters in Bernoulli percolation on transitive graphs cannot satisfy any non-trivial isoperimetric inequalities, but often satisfy the weaker anchored counterparts, which still have implications, e.g., on the behavior of random walk on the cluster: anchored $(2+\eps)$-dimensional isoperimetry implies transience \cite{Tho}, while anchored non-amenability implies an on-diagonal  heat kernel decay of $p_n(x,x)\leq \exp(-cn^{1/3})$ and positive speed of escape \cite{Virag}. For Bernoulli percolation, an affirmative answer to Question~\ref{q.HeHu}~(A) was conjectured in \cite{BLS}, while the connection between anchored isoperimetry and supercritical exponential decay was pointed out by the first author \cite{CPP,Pet}: a positive answer to Question~\ref{q.HeHu}~(B) implies a positive answer to question (A) for Bernoulli percolation, and more generally, for any independent perturbation of an invariant process.

The proof of property (B) by Hermon and Hutchcroft \cite{HH} seemed to use quite mildly the independence in Bernoulli percolation, hence it was reasonable to hope that the argument could generalize to any {\it  finite energy} ergodic invariant percolation, like most of the models mentioned above. Instead of defining here this ``finite energy'' condition precisely (also called uniform insertion and deletion tolerance; see \cite[Section 12.1]{PGG}), let us just assume a stronger version, as given by Question~\ref{q.HeHu}: the process is an independent perturbation of an invariant process. However, even in this setting, we will show that the answer to~(A), and hence also to~(B), is negative:

\begin{theorem}\label{main}
There exists an invariant percolation on the 4-regular tree with the property that for $\delta,\eps>0$ small enough, after adding a Bernoulli($\eps$) set of edges and removing a Bernoulli($\delta$) set of edges, conditioned on the component of a fixed vertex to be infinite, the cluster has no anchored expansion almost surely. 
\end{theorem}

Of course, this counterexample leaves it open whether standard finite energy invariant percolations, such as the FK random cluster model and the other models mentioned above, satisfy the properties in Question~\ref{q.HeHu}. 
\medskip

We assume that the reader is familiar with the notion of unimodular random rooted graphs; see \cite{AL, BC, PGG} for background. We recall the definition for the case of regular graphs, because it will be needed at one point of the proof which is not standard.

Consider an element of the form $(G,o;\alpha)$, where $G$ is a connected locally finite graph, $o$ is a distinguished vertex ({\it root}), and $\alpha$ is a subset of the edges $E(G)$, which can also be thought of as a {\it mark} on certain edges. Say that two such objects are equivalent, if there is a rooted graph isomorphism that takes one of them to the other and preserves the marks. Call the set of these equivalence classes $\Gstar$. In notation we will not distinguish between the equivalence class and a particular element representing it. Also, one may have more than one type of marks given, say $\alpha$ and $\beta$, in which case it will be convenient to list them all after the semicolon and write the element of $\Gstar$ as $(G,o;\alpha,\beta)$. For the sake of simplicity, by a slight abuse of notation, we will also use $\Gstar$ for this space of rooted graphs, when two distinguished subsets of edges are given as marks.
Now, let $\mu$ be a probability measure on $\Gstar$, and suppose that $G$ is regular $\mu$-almost surely. Then we call $\mu$ (or the random graph that it samples) {\it unimodular}, if for a uniformly chosen neighbor $x$ of $o$, the doubly rooted graph $(G,o,x;\alpha)$ has the same distribution as $(G,x,o;\alpha)$. See \cite{BC} for the equivalence of this definition and the more usual one, and also for the more general (nonregular) case, where a rebias by the degree of the root is needed.

\section{Construction of one component in an invariant percolation}

Let $\C$ be the canopy tree of degrees 1 and 4, a standard example of a unimodular random tree (see, e.g., \cite[Chapter 14]{PGG}), and a building block of many unimodular counterexamples \cite{BPT,souvlaki,AH}. Call the set $\LL_0$ of leaves {\it level $0$}, and the set $\LL_i$ of vertices at distance $i$ from $\LL_0$ {\it level} $i$. For every $v\in\LL_i$ and $j\geq 0$ there is a unique vertex $w\in\LL_{i+j}$ at distance $j$ from $v$. Call this vertex the $j${\it -grandparent} of $v$, and also say that $v$ is a $j${\it -grandchild} of $w$. 

Fix $p_i=4^{-i}$ for $i\in \N$. For every vertex $v$ of $\C$, where $v\in\LL_i$, define a Bernoulli($p_i$) random variable $\xi_v$, and let all the $\xi_v$ be independent from the others. For $x\in\LL_0$, define 
$$m(x):=\max \big\{ i : \xi_w=1\textrm{ for the $i$-grandparent $w$ of $x$}\big\}.$$
For every $x\in\LL_0$, define a finite ternary tree $T_x$ of depth $m(x)$ starting from root $x$ (that is, $x$ has degree 3 in $T_x$, every vertex at distance at most $m(x)-1$ has degree 4, and every vertex at distance $m(x)$ has degree 1). 
Let the $T_x$ be all disjoint from each other and from $\C$, apart from $x$. Say that a $y\in T_x$ has type $i$ if $m(x)=i$. 
Define the tree $\C^+:= \C \cup \bigcup_{x\in\LL_0} T_x$. Note that if we are only given $\C^+$, we can still identify $\C$ with probability 1. (Starting from an arbitrary leaf $x$ of $\C^+$, take the first vertex $y$ separating $x$ from infinity such that there is a subgraph of $\C^+$ that is isomorphic to $\C$ and has $y$ as a leaf.)
Extend the definition of $T_v$ to every $v\in V(\C)$ as the graph induced by the union of $\{v\}$ and all the finite components of $\C^+\setminus\{v\}$.

\begin{figure}[htbp]
\centerline{
\includegraphics[width=\textwidth]{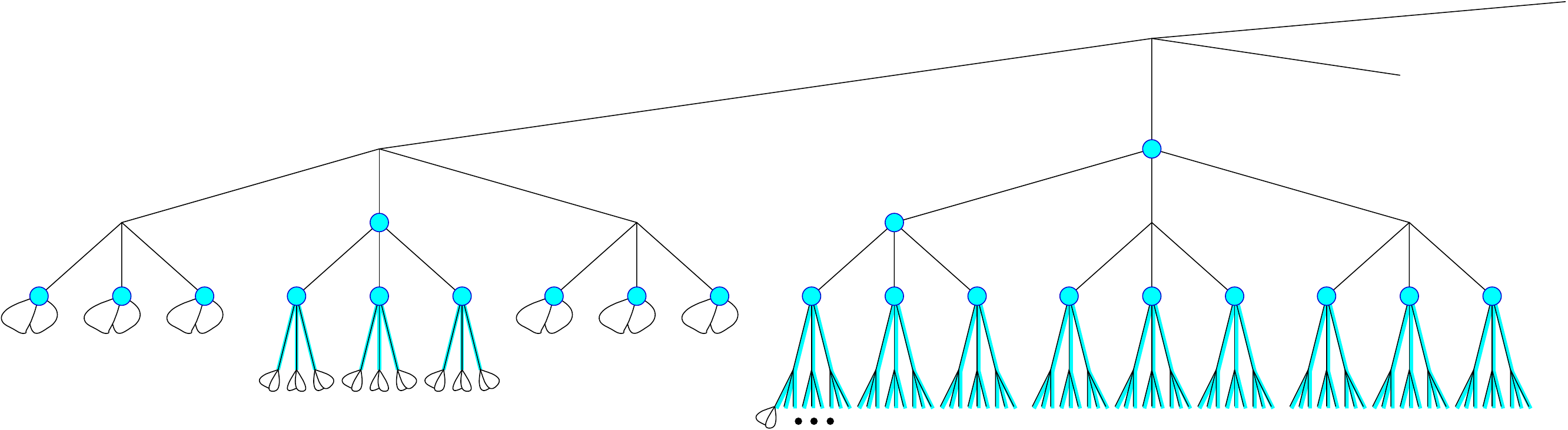}
}
\caption{Without the loops at the leaves, this is a copy of the tree $\C^+$, with the edges of $\C^+\setminus \C$ shown in turquoise. The vertices $v$ of $\C$ with $\xi_v=1$ are also colored turquoise. Together with the loops (where triples of loops are symbolized by the ``double-petals''), this is $\C^{++}$, whose colored covering tree, with the turquoise edges removed, is the invariant percolation $\omega$ on the 4-regular tree.}
\label{f.canoplus}
\end{figure}

The tree $\C$ can be turned into a unimodular random graph by picking the root to be a vertex in $\LL_i$ with probability proportional to $3^{-i}$. Using the fact that $\P(m(x) > i) < 4^{-i}$ holds for any $x\in\LL_0$, we have $\E (|T_x|)=\E (\sum_{i=0}^{m(x)} 3^{i})<\infty$, hence we can conclude that $(C^+,o)$ is also unimodular with a suitably chosen random root $o$; see Subsection 1.4 in \cite{BPT}.

\section{Construction of the invariant percolation}

Now we first construct an invariant percolation $\FF$ on the 4-regular tree $\TT$ where the component of a fixed root has the same distribution as the unimodular random graph that we constructed earlier. (It is tempting to apply \cite{BLS}, where a general such construction is given, but we will use a different argument because we want some extra properties to hold regarding the location of the components with respect to each other. Also, in our case all degrees of the unimodular graph are 1 or 4, making it easier to represent it as an invariant percolation.) Moreover, we will do it so that the components of the percolation will be isomorphic to each other: we first sample $\C^+$, and then fit infinitely many pairwise disjoint copies of this sample into $\TT$ in such a way that these copies cover every vertex of $\TT$. We will then use $\FF$ to define $\omega$, the invariant percolation of Theorem~\ref{main}.

Fix a random instance of $(\C^+,o)$. For a slick definition of $\FF$ and $\omega$, proceed as follows. To every leaf of $\C^+$, add 3 extra oriented loop-edges; let the set of all these loop-edges be $O$. The resulting random graph $\C^{++}:=\C^+\cup O$ is 4-regular, hence the 4-regular tree $\TT$ is the universal cover of $\C^{++}$. Fix a random covering map: a fixed vertex $o_{\TT} \in V(\TT)$ is mapped to the root $o\in V(\C^{++})$, then we choose a uniform permutation of the 4 neighbors of $o_{\TT}$ to be mapped to the four neighbors of $o$ (some of them might be $o$ again, through the loop edges in $O$), then a uniform permutation of the 3 further neighbors of each of the 4 neighbors in $\TT$ to be mapped to the 3 further neighbors in $\C^{++}$, and so on. Now, the preimage of $\C^+$ ($\subset \C^{++}$) in $\TT$ by this map is $\FF$, while the preimage of $\C\cup O$ ($\subset \C^{++}$) is $\omega$. 

We also give a more hands-on definition. Let $(\T_0,o):=(\C^+,o)$. We will construct a subgraph $\FF\subset \TT$ with the property that every component of $\FF$ is isomorphic to $\T_0$, and moreover, every non-singleton component of $\TT\setminus E(\FF)$ is a 3-regular tree $R$ with the property that for every $x_1,x_2\in V(R)$
and $\FF$-components $\F_{x_i}$ of $x_i$ ($i=1,2$), the $(\F_1,x_1)$ and $(\F_2,x_2)$ are rooted isomorphic. Starting from $\T_0$, we will add edges, some of them marked to belong to $\FF$ and some of them not (so they will belong to $\TT\setminus E(\FF)$). We will denote the $\FF$-component of a vertex $x$ by $\F_x$ and its component in $\TT\setminus E(\FF)$ by $\R_x$. In particular, $\F_o=\T_0$.

Let $L$ be the set of leaves in $\T_0$, and let $\T_1:=\T_0\cup\bigcup_{v\in L}\R_v$, where the $\R_v$ are pairwise disjoint 3-regular trees with one vertex being $v$ and all other vertices being outside of $\T_0$. 
Define $\T_2:= \T_1\cup\bigcup_{v\in L}\bigcup_{w\in \cup V(\R_v), w\not=v} \F_w$, where every $(\F_w,w)$ is rooted isomorphic to the $(\F_v,v)$ where $w\in V(\R_v)$. 
Similarly, define $\T_{2n+1}$ to be $\T_{2n}$ with a new 3-regular tree attached to every leaf of $\T_{2n}$. Define $\T_{2n}$ from $\T_{2n-1}$ by attaching a new tree $\F_w$ to every new vertex $w$ of each of these 3-regular trees, such that if $R$ is such a 3-regular tree and the vertex of it that is contained in $T_{2n-2}$ is $v$, then $(\F_w,w)$ is rooted isomorphic to the $(\F_v,v)$.
The limit of the $(\T_n,o)$ is a 4-regular rooted tree, which we identify with $(\TT,o)$ via an independent uniform random permutation at each vertex, just as in the universal cover construction. Every $\F_x$ is isomorphic to $\T_0=\F_o$, hence we can identify the canopy subgraph of it (which is preserved by any automorphism of $\F_x$ almost surely); call it $\Can_x$.
Let $\FF:=\bigcup_x \F_x$.
Finally we are ready to define the percolation process $$\omega:=\bigcup_x \big(\R_x\cup \Can_x\big)$$ 
as the union of edges that either belong to a canopy copy or to a regular tree copy.

Consider now the decorated rooted graph $(\TT,o;\omega,\FF)$ (here $\omega$ and $\FF$ are viewed as decorations). 

\begin{proposition}\label{unimogyula}
The decorated rooted random graph $(\TT,o;\omega,\FF)$ is unimodular.
In particular, $\omega$ and $\FF$ are invariant percolations on the 4-regular tree. Moreover, $\omega$ is ergodic.
\end{proposition}

\begin{proof}
We have obtained the decorated random rooted graph $(\C^{++},o;\C, O)$ from the unimodular random rooted graph $(\C^+,o;\C)$ via a very simple local modification that depends only on the rooted isometry class of the neighborhood of each vertex, hence it is also unimodular. Then, any instance of the random covering map from $\TT$ to $\C^{++}$ induces a natural measure preserving bijection between simple random walk paths on $\TT$ and on $\C^{++}$. Since the random walk criterion of unimodularity holds for $(\C^{++},o;\C, O)$, it also holds for $(\TT,o;\omega,\FF)$.

The second claim follows from Theorem 3.2 in \cite{AL}. 

For the ergodicity of $\omega$, first note that there is a measure-preserving bijection between $(\TT,o;\omega,\FF)$ and $(\C^+,o)$, and the action of each automorphism of $\TT$ on the former one induces just a rerooting in the latter one. Thus, if there was a non-trivial invariant property that $\omega$ satisfied, then the above bijection would translate it into a non-trivial rerooting-invariant property of $(\C^+,o)$, hence $(\C^+,o)$ would not be an extremal unimodular random rooted tree. However, $(\C^+,o)$ is in fact extremal (in other words, ergodic), for the following reason. Consider $\C$ in the construction of $\C^+$, and condition on the root being in $\C$, to obtain the decorated unimodular random graph $(\C,o;\C^+)$. The decoration here is a result of a factor of iid map, thus the decorated graph is also ergodic by the ``Decoration lemma'' (Lemma 2.2) of \cite{Ti} (which is stated for indistinguishability of percolation clusters, but is essentially the same as ergodicity of unimodular random graphs). We got $(\C,o;\C^+)$ from $(\C^+,o)$ (decorated with $\C$) by conditioning on $o\in\C$, hence the former is
absolutely continuous with respect to the latter, and it can be ergodic (extremal) only if the latter was also extremal.
\end{proof}

\section{Expansion properties of a component in the noised $\omega$-percolation}

For any $\eps>0$ and $\delta>0$, let $\eta^\eps$ and $\eta_\delta$ be independent Bernoulli bond percolation configurations on $E(\TT)$ of parameters $\eps$ and $\delta$ respectively. Define $\omps=(\omega\cup \eta^\eps)\setminus \eta_\delta$.
We will show that, if $\delta$ and $\eps$ are small enough, then the component of $o$ in $\omps$ is infinite and has no anchored expansion with positive probability. 


First we will examine the subgraph $\T_0$ as in the construction of the percolation; recall that $\T_0$ was sampled from $\C^+$, so by a slight abuse of notation we will identify the two and use references from the construction of $\C^+$.
Let $A$ be the event that $o$ is a leaf of type 0 in $\T_0$. We mention that a leaf of type 0 necessarily has to be in $\LL_0 \subset \C$. Conditioned on $A$, for every $n\in\N^+$ we will define a finite subgraph $H_n=H$ in $\omps\cap\FF$. Let $w(o)=w$ be the $3n$-grandparent of $o$. Let any $2n$-grandchild $v$ of $w$ be called {\it $n$-good} if $\xi_v=1$. If $v$ is $n$-good, then every $n$-grandchild of $v$ has type at least $n$, and thus the ternary subtree $T_v$ of $\C^+$ rooted in $v$ has depth at least $2n$: the first $n$ levels are in $\C$, and the remaining levels are in $\F_o\setminus \Can_o$. Let us denote by $A'_n$ the event that $A$ holds and at least one $n$-good $v$ exists. Note that
\begin{equation}\label{type}
\P (A'_n \given A) \geq 1-(1-4^{-n})^{3^{2n}} > 1-\exp(-(9/4)^n),
\end{equation}
and therefore, if $A'$ denotes the event that $\big\{A'_n$ occurs for all but finitely many $n \big\}$, then $\P( A' \given A ) = 1$. Note also that this $A'$ is dependent on the $\xi$-labels (or, in other words, on the actual sample $\T_0$ from $\C^+$), 
but not on $\eta^\eps$ or $\eta_\delta$. 

Condition on $A'$, and let $v=v(o)$ be an $n$-good vertex. 

Let $t_v$ be the $\omps$-component of $v$ in $T_v$ up until generation $n$. Conditioned on $A'$, this is the first $n$ generations of a branching process, whose mean offspring is $3(1-\delta)$. We will need a small large deviations lemma for such branching processes. It follows, for instance, from \cite{AN}, but we include here a direct proof for the sake of completeness.

\begin{lemma}\label{GWLD}
Consider a branching process $(Z_n)_{n=1}^\infty$ with offspring
distribution $X$ that has expectation $\mu>1$ and variance
$\sigma<\infty$. Fix any $\kappa \in (1,\mu)$. Then, there exists
$\lambda=\lambda(\mu,\sigma,\kappa)>0$ such that
$$
\P \big( Z_n < \kappa^n \bgiven Z_m>0 \textrm{ for all } m\ge 0 \big)
< \exp(-\lambda n),
$$
for all $n$ large enough. If $X \sim \mathsf{Binom}(3,1-\delta)$ and
$\kappa \in (1,3)$ is fixed, then $\lambda$ can be made arbitrarily large
by taking $\delta$ small enough.
\end{lemma}

Let us remark that one can not generally get a bound that is better than exponential in $n$. For instance, in the case of $X\sim \mathsf{Binom}(3,1-\delta)$, the first $\alpha n$ generations for any $\alpha\in (0,1)$ could always be just one child, which happens with an exponentially small probability and reduces the size of $Z_n$ by an exponential factor.

\begin{proof} 
We start by recalling a much weaker bound, using just the second moment method (see, e.g., \cite[Exercise 12.12]{PGG}). The first moment is $\E Z_n = \mu^n$. Regarding the variance, 
$$\Var(Z_n) = \E \big( \Var(Z_n \given Z_{n-1}) \big)+ \Var \big( \E (Z_n \given Z_{n-1}) \big)= \sigma^2 \mu^{n-1} + \mu^2 \Var(Z_{n-1}).$$ 
Writing $\gamma_n := \Var (Z_n) / \mu^{2n}$, we get the recursion $\gamma_n = \sigma^2 / \mu^{n+1} + \gamma_{n-1}$, and thus $\lim_{n\to\infty} \gamma_n = \sigma^2 / (\mu^2-\mu)$. Therefore, $\Var(Z_n)\sim\sigma^2 \mu^{2n} / (\mu^2-\mu) \asymp (\E Z_n)^2$. By the Paley-Zygmund inequality, there exists $b=b(\mu,\sigma)>0$ such that, for all $n\ge 0$,
\begin{equation}\label{PZ}
\P(Z_n \geq b \mu^n) \geq b\,.
\end{equation}
When $X\sim \mathsf{Binom}(3,1-\delta)$ with $\delta$ small, then $\sigma$ is small, hence, using Chebyshev's inequality instead of Paley-Zygmund, we can make $b$ arbitrarily close to 1. 

Consider now the depth-first exploration of the tree, as in
\cite[Figure 12.4]{PGG}: when a vertex in the depth-first order is
examined, its entire offspring is revealed. Under the conditioning
that the tree is infinite, every generation $i\ge 0$ has a last time
when it is visited by this exploration; denote by $v_i$ the vertex at
which this happens, and by $X_i$ the offspring of $v_i$. Let $E_i$
denote the event that $X_i \ge 2$ and the first child of $v_i$ that
gets examined by the exploration has an infinite offspring. In this
case, the exploration leaves at least one child of $v_i$ unexplored. 
If we denote the sigma-algebra generated by $\{ E_i : i=0,1,\dots,j\}$ by $\cE_j$, then 
$$
\P(E_{j+1} \given \cE_j) \geq q := \P(X \ge 2) \,  \P(Z_m>0 \ \forall m\ge 0) > 0.
$$
When $X\sim \mathsf{Binom}(3,1-\delta)$, this $q$ converges to 1 as $\delta\to 0$. Therefore, for any $\alpha\in (0,1)$, the probability that out of $\{ E_i, \ i=0,1,\dots,\alpha n\}$ less than $\alpha n q/ 2$ events will occur is smaller than 
\begin{equation}\label{binom}
\P\big(\mathsf{Binom}(\alpha n,q) < \alpha n q/2\big) \le \exp(-c n),
\end{equation} 
with some $c=c(\alpha,q)>0$, which goes to infinity as $q\to 1$. Let $I$ be the set of indices $i \in  \{0,1,\dots,\alpha n\}$ for which $E_i$ occurs, and condition on the event $\big\{ | I | \ge \alpha n q/2 \big\}$.

For each $i\in I$, denote the progeny of the unexplored second child by $\big\{ Z^{(i)}_j : j\ge 0 \big\}$. The main idea is that 
\begin{equation}\label{Zsum}
Z_n \geq \sum_{i\in I} Z^{(i)}_{n-1-i}\,,
\end{equation} 
where the summands are independent. By~(\ref{PZ}), we have $\P \big( Z^{(i)}_{n-1-i} < \kappa^n\big) < 1-b$ if $\kappa^n \leq b \mu^{n-1-i}$, which does hold for all $i\leq \alpha n$ whenever $\alpha>0$ is small enough so that $\kappa < \mu^{1-\alpha}$, and when $n$ is large enough. Combining~(\ref{PZ}),~(\ref{binom}), and~(\ref{Zsum}),
$$
\P( Z_n < \kappa^n ) \leq \exp(-cn) + (1-b)^{\alpha n q / 2} < \exp(-\lambda n),
$$
for some $\lambda > 0$ and all $n$ large enough. For $X\sim \mathsf{Binom}(3,1-\delta)$ as $\delta\to 0$, we have $c\to\infty$, $b\to 1$, and $q\to 1$, while $\alpha$ is fixed by $\kappa$ and $\mu$, hence we can take $\lambda\to\infty$.
\end{proof}
\medskip

Getting back to the analysis of our process, Lemma~\ref{GWLD} implies that, for $\delta>0$ small enough,
\begin{equation}\label{tree}
\P \bigl(|t_v|< 2^n \bgiven A' \bigr)\leq 2^{-n}.
\end{equation}
If there is a path from $v$ to a leaf in $T_v$ then define $t_v^+:=t_v$, otherwise define $t_v^+$ as the $\omega^\eps_\delta$ component of $v$ in $T_v$. Since $t_v\subset t_v^+$, \eqref{tree} remains valid with $t_v$ replaced by $t_v^+$.
Let $P_v$ be the path between $o$ and $v$. We have
\begin{equation}\label{path}
\P (P_v \subset\omps \given \omega )\geq (1-\delta)^{5n},
\end{equation}
for almost every $\omega\in A'$.
Putting these together, we obtain that, conditioned on $A'$, for $\delta$ small enough,
with probability at least $(1-\delta)^{5n}-2^{-n}>4^{-n}$,
we have that $P_v\subset \omps$ and $|t_v^+|>2^n$ . Call this event $B_n$. Conditioned on $B_n$, define $H:=P_v\cup t_v^+ \subset \omega^\eps_\delta$. 
We have just seen $\P(B_n \given A')> 4^{-n}$ and that, conditioned on $B_n$,
\begin{equation}\label{H}
|H|\geq 5n+2^n.
\end{equation}
Next we find an upper bound on the size of the boundary $\partial H$ of $H$ {\it inside} $\omps$. 
For any fixed $u$ of $T_v\cap\LL_0$, the probability that $T_u\cap \omps$ has a path from $u$ to distance $n$ in $T_v\setminus E(\C)$ is trivially bounded by $\eps^n 3^n$. 
If this event does not happen for any $u$, then the boundary of $t_v^+$ in $\omps$ consists only of the single edge of $P_v$ incident to $v$. By a union bound we conclude 
\begin{equation}\label{boundary}
\P\big(|\partial H|> 2|P_v|+1 \bgiven B_n \bigr)\leq \P \big( |\partial t^+_v|>1 \bgiven  B_n \big)
\leq \eps^n 3^n |T_v\cap\LL_0|\leq \eps^n 9^n.
\end{equation}
To summarize, for $\delta,\eps$ small enough we have just obtained the following:

\begin{proposition}\label{propolisz}
Conditioned on $A'$, for all but finitely many $n\in\N^+$, 
with probability at least $c_n:=\P(B_n \given A')-\eps^n 9^n\geq 8^{-n}$ 
there exists an $H\subset\omps\cap\FF$
such that $o\in H$, and 
\begin{equation}
|\partial H|\leq 10n+1
\,\,\,\,\,\,{\text{ and }}\,\,\,\,\,\,
|H|\geq 5n+2^n.
\end{equation}
\end{proposition}

Say that $o$ is $n${\it -nice} (or just {\it nice}), when $H$ as in Proposition \ref{propolisz} exists.
Recall the construction of $\T_1$ and $\T_2$: if we pick any vertex $x$ of $\R_o$, then $(\F_x,x)$ is rooted isomorphic to $(\F_o,o)$. Hence conditioning on $A'$ means that an analogous event holds for $x$ as well, namely, $x$ is a leaf (of type 0) in $\F_x$ and there is an $n$-good vertex for it. Hence we can define $x$ to be $n$-nice just the way we defined it for $o$. Moreover, for all $x\in V(\R_o)$ the events of $x$ being nice are conditionally independent of each other, because they are determined by $\eta^\eps$ and $\eta_\delta$ on disjoint edge sets (the $\F_x$).

Now consider $\CC:=\R_o\cap\omps$. Let $E$ be the event that $o$ is in an infinite component of $\CC$, and let $\BB_r$ be the ball of radius $r$ around $o$ in this component. By Lemma~\ref{GWLD}, we have that
$\P(|\BB_r|> 2^r \given E)$ tends to 1, exponentially fast in $r$. Then 
\begin{align*}
\P\big( \textrm{there is no $n$-nice point in } \BB_r  \bgiven E\cap A' \big) \leq \P\big( |\BB_r|< 2^r \bgiven E\cap A' \big)+(1-c_n)^{2^r},
\end{align*}
with $c_n \ge 8^{-n}$ from the proposition. Choosing $r=r_n=n^2$, say, this quantity tends to 0, superexponentially fast in $n$. We can conclude that on $E\cap A'$, almost surely for all but finitely many $n\in\N^+$ there is an $n$-nice vertex $x_n$ at distance at most $n^2$ from $o$ in $\CC$. Then consider the $H=H_n(x_n)$ from Proposition~\ref{propolisz} that corresponds to this $x_n$ in $\F_{x_n}$. Let $Q_n$ be the path in $\CC$ between $o$ and $x_n$, and define $K_n=Q_n\cup H_n$. Then, using the proposition, we have
\begin{equation}
|\partial K_n|\leq 12n+n^2
\,\,\,\,\,\,{\text{ and }}\,\,\,\,\,\,
|K_n|\geq 5n+2^n.
\end{equation}
This shows that $K_n$ is an anchored F{\o}lner sequence in $\omps$, i.e., satisfies $|\partial K_n|/|K_n|\to 0$, finishing our proof.

%
%

 \section*{Acknowledgements}
This research was partially supported by the ERC Consolidator Grant 772466 ``NOISE''. The second author was also supported by Icelandic Research Fund, Grant Number: 185233-051.


\begin{thebibliography}{}

\bibitem{AC}
A. Ab\"acherli and J. \v{C}ern\'y.
Level-set percolation of the Gaussian free field on regular graphs I: regular trees.
{\it Electron. J. Probab.} {\bf 25} (2020), paper no. 65, 24 pp.

\bibitem{ABF}
M. Aizenman, D. J. Barsky, and R. Fern\'andez. 
The phase transition in a general class of Ising-type models is sharp. 
{\it J. Statist. Phys.} {\bf 47} (1987), 343--374.

\bibitem{AL}
D. Aldous and R. Lyons.
Processes on unimodular random networks.
{\it Electron. J. Probab.} {\bf 12} (2007), 1454--1508.

\bibitem{AH}
O. Angel and T. Hutchcroft.
Counterexamples for percolation on unimodular random graphs. 
{\it Unimodularity in randomly generated graphs}, pp.~11--28. 
Contemp. Math., 719, Amer. Math. Soc., Providence, RI, 2018. 

\bibitem{AN}
K.~B. Athreya and P.~E. Ney. 
The local limit theorem and some related aspects of supercritical branching processes.
{\it Trans. Amer. Math. Soc.} {\bf 152} (1970), 233--251.

\bibitem{BV}
\'A. Backhausz and B. Vir\'ag.
Spectral measures of factor of i.i.d. processes on vertex-transitive graphs.
{\it Ann. Inst. H. Poincar\'e (B) Probab. Statist.} {\bf 53} (2017), 2260--2278.

\bibitem{BST}
A. Bandyopadhyay, J. Steif, and \'A. Tim\'ar. 
On the cluster size distribution for percolation on some general graphs.
{\it Rev. Mat. Iberoam.} {\bf 26} (2010), 529--550.

\bibitem{BC}
I. Benjamini and  N. Curien.
Ergodic theory on stationary random graphs.
{\it Electron. J. Probab.} {\bf 17} (2012), paper no.~93, 1--20.

\bibitem{pertu} I. Benjamini, R. Lyons and O. Schramm.
Percolation perturbations in potential theory and random
walks, {\it In: Random walks and discrete potential theory (Cortona,
1997)}, Sympos. Math. XXXIX, Cambridge Univ. Press, 1999, pp. 56--84.

\bibitem{BLS}
I. Benjamini, R. Lyons and O. Schramm.
Unimodular random trees. 
{\it Ergod. Th. Dynam. Sys.} {\bf 35}  (2015), 359--373.

\bibitem{BS} 
I. Benjamini and O. Schramm. 
Percolation beyond $\Z^d$, many questions and a few answers. 
{\it Elect. Commun. Probab.} {\bf 1} (1996), 71--82. 

\bibitem{BPT}
D. Beringer, G. Pete and \'A. Tim\'ar,
On percolation critical probabilities and unimodular random graphs. 
{\it Electron. J. Probab.} {\bf 22} (2017), paper no.~106. 

\bibitem{souvlaki}
J. Carmesin, B. Federici and A. Georgakopoulos, appendix by G. Pete and G. Ray.
A Liouville hyperbolic souvlaki. Appendix: A unimodular Liouville hyperbolic souvlaki.
{\it Elect. J. Probab.} {\bf 22} (2017), paper no.~36, pp. 19.

\bibitem{CPP}
D. Chen, Y. Peres and G. Pete, 
Anchored expansion, percolation and speed.
{\it Ann. Probab.} {\bf 32}  (2004), 2978--2995.

\bibitem{GFF}
H. Duminil-Copin, S, Goswami, P-F. Rodriguez and F. Severo.
Equality of critical parameters for percolation of Gaussian free field level-sets.
\arXiv{2002.07735}{math.PR}

\bibitem{DCRT}
H. Duminil-Copin, A. Raoufi, and V. Tassion. Sharp phase transition for the random-cluster and Potts models
via decision trees. 
{\it Ann. of Math. (2)} {\bf 189} (2019), 75--99.

\bibitem{DCT} 
H. Duminil-Copin and V. Tassion. 
A new proof of the sharpness of the phase transition for Bernoulli percolation and the Ising model. 
{\it Comm. Math. Phys.} {\bf 343} (2016), 725--745.

\bibitem{HH}
J. Hermon and T. Hutchcroft.
Supercritical percolation on nonamenable graphs: Isoperimetry, analyticity, and exponential decay of the cluster size distribution. {\it Inv. Math.}, to appear. \arXiv{1904.10448v3}{math.PR}

\bibitem{TomIsing}
T. Hutchcroft.
Continuity of the Ising phase transition on nonamenable groups.
\arXiv{2007.15625}{math.PR}

\bibitem{KZh}
H. Kesten and Y. Zhang. 
The probability of a large finite cluster in supercritical Bernoulli percolation. 
{\it Ann. Probab.} {\bf 18} (1990), 537--555.

\bibitem{Lyons}
 R. Lyons. 
Phase transitions on nonamenable graphs.
{\it J. Math. Phys.} {\bf 41} (2000), 1099--1126. 

\bibitem{Lfiid}
 R. Lyons. 
Factors of IID on trees.
{\it Combinatorics, Probability \& Computing} {\bf 26} (2017), 285--300.

\bibitem{LP}
R. Lyons and Y. Peres. {\it Probability on Trees and Networks.}
Cambridge University Press, New York, 2016. Available at \url{http://pages.iu.edu/\textasciitilde rdlyons/}

\bibitem{Pet}
G. Pete. 
A note on percolation on $\Z^d$: isoperimetric profile via exponential cluster repulsion. 
{\it Electron. Commun. Probab.} {\bf 13} (2008), 377--392.

\bibitem{PGG}
G. Pete.
{\it Probability and Geometry on Groups}.
Book in preparation, \url{http://www.math.bme.hu/~gabor/PGG.pdf}

\bibitem{TT} 
A. Teixeira and J. Tykesson.
Random interlacements and amenability.
{\it Ann. Appl. Prob.} {\bf 23} (2013), 923--956.

\bibitem{Ti}
\'A. Tim\'ar. A nonamenable ``factor'' of a Euclidean space. {\it Ann. Probab.}, to appear.
\arXiv{arXiv:1712.08210}{math.PR}

\bibitem{Tho}
C. Thomassen. 
Isoperimetric inequalities and transient random walks on graphs. 
{\it Ann. Probab.} {\bf 20} (1992), 1592--1600.

\bibitem{Virag} B. Vir\'ag.
Anchored expansion and random walk.
 {\it Geom. Funct. Anal.} {\bf 10} (2000), 1588--1605.


\end{thebibliography}
\end{document}